\documentclass[12pt,amsfonts]{amsart}
\usepackage{amsmath}
\usepackage[all]{xy}
\usepackage{amssymb}
\usepackage{amsthm}
\usepackage[top=3.5cm, bottom=3.5cm, left=2.5cm, right=2.5cm, includefoot]{geometry}
\usepackage{amscd}
\usepackage[all]{xy}
\usepackage[pdftex,colorlinks=true,backref=page]{hyperref}
\theoremstyle{definition}
\newtheorem{dfn}{Definition}[section]
\newtheorem{thm}[dfn]{Theorem}
\newtheorem{lem}[dfn]{Lemma}
\newtheorem{cor}[dfn]{Corollary}
\newtheorem{pro}[dfn]{Proposition}
\newtheorem{exa}[dfn]{Example}
\newtheorem{rem}[dfn]{Remark}

\newcommand{\Z}{\mathbb{Z}}
\newcommand{\Q}{\mathbb{Q}}

\newcommand{\C}{\mathbb{C}}

\newcommand{\E}{\mathcal{E}}

\title{Self-Closeness Numbers of Non-Simply-Connected Spaces}
\author{TONG YICHEN}
\subjclass[2010]{55S37,55P10}
\keywords{self-closeness number, (co-)$H_0$-space, obstruction theory}
\address{Department of Mathematics, Kyoto University, Kyoto, 606-8502, Japan}
\email{tong.yichen.25m@st.kyoto-u.ac.jp}
\begin{document}
	\bibliographystyle{amsplain}
	\maketitle
	\begin{abstract}
		The self-closeness number $N\E(X)$ of a space $X$ is the least integer $k$ such that any self-map is a homotopy equivalence whenever it is an isomorphism in the $n$-th homotopy group for each $n\le k$. We discuss the self-closeness numbers of certain non-simply-connected $X$ in this paper. As a result, we give conditions for $X$ such that $N\E(X)=N\E(\widetilde{X})$, where $\widetilde{X}$ is the universal covering space of $X$. 
	\end{abstract}
	\section{Introduction}
	Given a self-map $f:X\rightarrow X$, one of the most fundamental problems in algebraic topology is to find the condition for the map being a homotopy equivalence. The standard answer is the J.H.C. Whitehead's first theorem, saying that if a map induces isomorphisms in homotopy groups in all dimensions and $X$ is a CW complex, then $f$ is a homotopy equivalence. However, it often happens that if we can see that $f$ induces isomorphisms in homotopy groups in some dimensions, instead of all dimensions, then we get that $f$ is a homotopy equivalence. For example, a self-map of $\C P^m$ is a homotopy equivalence if it induces an isomorphism in $\pi_2(\C P^m)$ since $\C P^m$ is simply-connected and its cohomology is generated by an element of degree 2. So it is natural to find the upper bound of such number. In 2015, Ho Won Choi and Kee Young Lee introduced the self-closeness number in \cite{CL15} as follows 
	\begin{dfn}
		The \emph{self-closeness number} of a space $X$, denoted by $N\mathcal{E}(X)$, is defined to be the least integer $k$ such that any self-map $f\colon X\to X$ is a homotopy equivalence whenever $f$ is an isomorphism in the $n$-th homotopy group for each $n\le k$.
	\end{dfn}
	By the J.H.C Whitehead's first theorem, this number represents the complexity of the topology of a space $X$, and is a relatively new invariant. 
	By the J.H.C. Whitehead's second theorem, the self-closeness number of a simply-connected space is easier to compute because it is controlled by homology. In \cite{OY17} and \cite{OY18}, some inequalities were obtained, which reveal the behavior of self-closeness numbers of simply-connected spaces under cell attachments or extensions by fibrations. However, these results are not applicable to non-simply-connected spaces and the only way to compute its self-closeness number is a straight forward calculation. In fact, the self-closeness numbers of non-simply-connected spaces have barely been studied and determined only for real projective spaces and lens spaces. The aim of this paper is to investigate the behavior of self-closeness numbers for some non-simply-connected spaces.

	The group of self homtopy equivalences $\mathcal{E}\left(X\right)$ of a space $X$ has been studied intensely for a long time. However, there is few researches on $\E(X)$ for $X$ non-simply-connected. The self-closeness number provides a powerful tool to study $\E(X)$ by focusing on isomorphisms in lower dimensions. The self-closeness numbers for several non-simply-connected spaces are computed by the following results, and these would help determine $\E(X)$ of them.

	Now we state our main theorems. For a graded algebra $A$, let $d(A)$ denote the maximal degree of generators of $A$. Recall that an $H_0$-space is a simply-connected finite complex which is an $H$-space after rationalization. 
	
	\begin{thm}
		\label{main}
		Let $X$ be a finite complex such that $\pi_1(X)$ is finite and its universal covering space $\widetilde{X}$ is an $H_0$-space. If $d(H^*(\widetilde{X};\mathbb{Z}))=d(H^*(\widetilde{X};\mathbb{Q}))$ and $\pi_1(X)$ acts trivially on $H^*(\widetilde{X};\mathbb{Q})$, then
		\[
		N\mathcal{E}(X)=N\mathcal{E}(\widetilde{X})=d(H^*(\widetilde{X};\Z)).
		\]
	\end{thm}
	We will also prove a dual version of Theorem \ref{main}. Recall that a co-$H_0$-space is a simply-connected finite complex which is a co-$H$-space after rationalization. Recall that the cohomological dimension of a path-connected space $X$ is defined by
	\[
	cd(X)=\sup\{d\mid H^d(X;M)\ne 0\text{ for some }\Z\pi_1(X)\text{-module }M\}
	\]
	\begin{thm}
		\label{main co-H}
		Let $X$ be a finite complex such that $\pi_1(X)$ is finite and its universal covering space $\widetilde{X}$ is an co-${H}_0$-space. If $cd(X)=d(H^*(X;\mathbb{Q}))$ and $\pi_1(X)$ acts trivially on $H^*(\widetilde{X};\mathbb{Q})$, then
		\[
		N\mathcal{E}(X)=N\mathcal{E}(\widetilde{X})=cd(X).
		\]
	\end{thm}
	Our result in Theorem \ref{main} applies to various spaces, and we will see them in section 2.

	The paper is organized as follows. Section 2 provides various examples of non-simply-connected spaces to which Theorem \ref{main} applies. Section 3 gives an upper bound of the self-closeness number of a non-simply-connected space by its universal covering space. Section 4 recalls the $p$-universality of simply-connected spaces. We also recall and show the properties of (co-)$H_0$-spaces which we will use in following sections. In Section 5 and Section 6, we prove Theorem \ref{main} and \ref{main co-H} by applying the obstruction theory for non-simply-connected spaces. 
	
	\section{Examples}
	In this section, we give several examples of non-simply-connected spaces whose universal covering spaces are $H_0$-spaces. We apply Theorem \ref{main} to compute their self-closeness numbers. 
	\subsection{Topological spherical space forms}
	Recall that a topological spherical space form is the quotient space of a sphere $S^n$ under a free action of a finite group. In \cite{KO20} Kishimoto and Oda determined the monoid structure of $[X,X]$ for $X$ a topological spherical space form, and the following example is a simple corollary of their result.

	For a even dimensional sphere,  the only free action by a finite group on $S^n$ is the $\Z/2$-action and it is not orientation preserving, so we cannot apply either Theorem \ref{main} or Theorem \ref{main co-H}.

	For $n$ odd, notice that $S^n$ is an $H_0$-space. Furthermore, by \cite{KO20}, every orientation reserving self-map of $S^n$ has a fixed point by Lefschetz fixed point theorem, so all free actions of finite groups on $S^n$ are orientation preserving, hence the induced action on $H^*(S^n;\Q)$ is trivial.

	From the discussion above we deduce that 
	\begin{pro}
		For a finite group $G$ acts freely on $S^n$ with $n$ odd, we have 
		\[N\E(S^n/G)=N\E(S^n)=n\]
	\end{pro}
	\subsection{Product of spheres} 
	A product of finitely many odd spheres \(X=S^{2 n_{1}+1} \times \cdots \times S^{2 n_{r}+1}\) for \(n_{i} \geq 1\) is an \(H_{0}\)-space as we will see in Theorem \ref{H_0} below. We show examples of free actions of finite groups on the product space \(X\), which are trivial in rational cohomology. Then we can apply Theorem \ref{main} to the quotient space.

	First, we consider a product of two spheres. Let \(\Sigma_{3}\) denote the symmetric group in three letters. Blaszczyk showed that \cite[Proposition 5.1]{Blaszczyk13} \(\Sigma_{3}\) acts freely on \(S^{m} \times S^{n}\) if and only if \(m\) or \(n\) is odd. Assume that $m=2k+1$, we explicitly give a free \(\Sigma_{3}\) -action on \(S^{2 k+1} \times S^{n}\). We regard \(S^{2 k+1}\) as the unit sphere in \(\mathbb{C}^{k+1}\). Then \(\mathbb{Z} / 3\) acts on \(S^{2 k+1}\) by
	
	\[
	x \mapsto e^{\frac{2 \pi \sqrt{-1}}{3}} x
	\]
	
	for \(x \in S^{2 k+1}\). We also regard \(S^{n}\) as the unit sphere in \(\mathbb{R}^{n+1}\). Then \(\mathbb{Z} / 2\) acts on both $S^{2k+1}$ and \(S^{n}\) antipodally. Now we recall that \(\Sigma_{3}\) has a presentation
	
	\[
	\Sigma_{3}=\left\langle a, b \mid a^{3}=b^{2}=(a b)^{3}=1\right\rangle
	\]
	
	Then we get an action of \(\Sigma_{3}\) on \(S^{2 k+1} \times S^{n}\) by combining the above actions of \(\mathbb{Z} / 3\) on \(S^{2 k+1}\) and \(\mathbb{Z} / 2\) on $S^{2k+1}\times S^n$ antipodally. Clearly, this action is free, and is trivial in rational cohomology whenever \(n\) is odd. Thus by Theorem \ref{main}, we obtain:
	
	\begin{pro}
		Let \(\Sigma_{3}\) acts on \(S^{m} \times S^{n}\) for \(m, n\) odd as above. Then
		
		\[
		N \mathcal{E}\left(\left(S^{m} \times S^{n}\right) / \Sigma_{3}\right)=N \mathcal{E}\left(S^{m} \times S^{n}\right)=\max \{m, n\}
		\]
	\end{pro}
	Let $D_{2q}$ be the dihedral group of order $2q$. Then it has a presentation 
	\[D_{2q}=\left\langle a, b \mid a^{q}=b^{2}=1,(ab)^2=1\right\rangle.\]
	Then in particular, $D_{6}=\Sigma_{3}$. We can easily generalize the above free action of $\Sigma_{3}$ on $S^{2k+1}\times S^n$ to a free action of $D_{2q}$ on $S^{2k+1}\times S^n$. Then we get: 
	\begin{cor}
		Consider the above free action of $D_{2q}$ on $S^{m}\times S^n$ for $m,n$ odd. Then 
		\[N\E((S^m\times S^n)/D_{2q})=N\E(S^m\times S^n)=\max\{m,n\}\]
	\end{cor}
	
	Next we consider a product of spheres $X=S^{2k+1}\times S^{2k+3}\times\cdots\times S^{2n-1}$ for $k\le n$. Suppose a finite subgroup $G\subset U(n)$ of order $m$ acts freely on $U(n)/U(k)$. If $m$ is prime to $(n-1)!$, Adem, Davis and \"{U}nl\"u \cite[Theorem 1.4]{ADU04} proved that $G$ acts freely and smoothly on $X=S^{2k+1}\times S^{2k+3}\times\cdots\times S^{2n-1}$. In addition, it acts trivially on $H_*(X;\Z[\frac{1}{m}])$, where $\Z[\frac{1}{m}]$ is a subring of $\Q$ consisting of rational numbers whose denominators are power of $m$. This implies that induced action of $G$ on $H^*(X;\Q)$ is trivial too. A example of such a finite group is given explicitly by Adem, Davis and \"Unl\"u in \cite[Theorem 1.5]{ADU04}. By applying Theorem \ref{main}, we obtain:  
	\begin{pro} 
		Let $G$ and $X$ be as above. Then
		\[N\E(X/G)=N\E(X)=2n-1\] 
	\end{pro}
	The main theorem in \cite{DM91} proved by Davis and Milgram also provides various examples of free actions of finite groups on $H_0$-spaces such that we could apply Theorem \ref{lower bound H_0} and Corollary \ref{upper bound 3}.

	\subsection{Lie groups}
	We consider a finite subgroup $G$ of $SU(n)$. Then $G$ acts freely on $SU(n)$ by the multiplication, and since this action factors through the path-connected space $SU(n)$, it induced action in rational cohomology is trivial. On the other hand, since $H^*(SU(n);\Z)=\Lambda_{\Z}(x_3,...,x_{2n-1})$ where $|x_i|=i$, we have $d(H^*(SU(n);\Z))=d(H^*(SU(n);\Q))=2n-1$. Then by applying Theorem \ref{main}, we have: 
	\begin{pro}
		For any finite subgroup $G\subset SU(n)$ acts on $SU(n)$ by the multiplication, we have 
		\[N\E(SU(n)/G)=N\E(SU(n))=2n-1\]
	\end{pro}
	Now we generalize this result to the complex Stiefel manifold $V_k(\C^n)=U(n)/U(n-k)$. Consider a finite subgroup $G$ of $U(n)$. Then $G$ acts on $U(n)/U(n-k)$ by the multiplication, and since this action factors through the path-connected space $U(n)$, its induced action in rational cohomology is trivial. Furthermore, such an action is free if and only if for any $M\in U(n)$, $M^{-1}GM$ intersects $U(n-k)$ trivially in $U(n)$. On the other hand, since $H^*(V_k(\C^n))=\Lambda_{\Z}(x_{2(n-k)+1},x_{2(n-k)+3},...,x_{2n-1})$ where $|x_i|=i$, $d(H^*(V_k(\C^n);\Z))=d(H^*(V_k(\C^n);\Q))=2n-1$. Thus we have:  
	
	\begin{pro}
		For any finite subgroup $G\subset U(n)$ such that $M^{-1}GM$ intersects $U(n-k)$ trivially in $U(n)$ for any $M\in U(n)$, we have  
		\[N\E(V_k(\C^n)/G)=N\E(V_k(\C^n))=2n-1.\]
	\end{pro}
	All of the above results hold if we replace $U(n)$ with $Sp(n)$. 
	\section{Upper bounds}
	
	In this section, we give an upper bound of the self-closeness number of a non-simply connected space by that of its universal covering space. Then we give a cohomological upper bound for the self-closeness number of a simply-connected space.

	First, we give an upper bound of the self-closeness number of a non-simply connected space.

	\begin{pro}
		\label{upper bound 1}
		Let $X\to Y$ be a covering over a path-connected base $Y$. Then
		\[
		N\mathcal{E}(Y)\le\max\{N\mathcal{E}(X),1\}.
		\]
	\end{pro}
	
	\begin{proof}
		Suppose that a map $f\colon Y\to Y$ is an isomorphism in $\pi_1$. Let $\tilde{f}\colon X\to X$ denote a lift of $f$. Suppose that $\tilde{f}$ is an isomorphism in $\pi_*$ for $*\le n$. Then by the naturality of the homotopy exact sequence, $f$ is also an isomorphism in $\pi_*$ for $*\le\max\{n,1\}$, completing the proof.
	\end{proof}

	

	\begin{cor}
		\label{upper bound 3}
		Suppose that a path-connected space $X$ is of finite dimension and $\pi_1(X)$ is finite. Then
		\[
		N\mathcal{E}(X)\le N\mathcal{E}(\widetilde{X})
		\]
		where $\widetilde{X}$ denotes the universal cover of $X$.
	\end{cor}
	
	\begin{proof}
		By Proposition \ref{upper bound 1}, it is sufficient to show $N\mathcal{E}(\widetilde{X})\ne 0$, or equivalently, $\widetilde{X}$ is not weakly contractible. Indeed, if $\widetilde{X}$ is weakly contractible, then $X$ is of the homotopy type of $K(\pi_1(X),1)$. Since $\pi_1(X)$ is finite, $K(\pi_1(X),1)$ is of infinite dimension, a contradiction.
	\end{proof}

	Oda and Yamaguchi \cite{OY20} introduced the homological self-closeness number $N_*\mathcal{E}(X)$ by replacing the homotopy groups in the definition of the self-closeness number with homology. It is immediate from the J.H.C. Whitehead theorem that
	\begin{equation}
		\label{homotopy vs homology}
		N\mathcal{E}(X)=N\mathcal{E}_*(X)
	\end{equation}
	whenever $X$ is simply-connected and of finite type. In particular, if $X$ is a simply-connected finite complex, then $N_*\mathcal{E}(X)$ is finite, implying $N\mathcal{E}(X)$ is finite too. By Proposition \ref{upper bound 1}, we can generalize this finiteness of the self-closeness number to a non-simply connected space.
	
	\begin{cor}
		If the universal covering space of a path-connected space $X$ is a finite complex, then $N\mathcal{E}(X)$ is a finite integer.
	\end{cor}

	Recall from \cite{OY20} that for a space $X$, the cohomological self-closeness number $N^*\mathcal{E}(X)$ is defined to be the least integer $n$ such that a self-map $f\colon X\to X$ is an isomorphism in cohomology whenever it is an isomorphism in cohomology of dimension $\le n$. Now we give an upper bound if the self-closeness number by the cohomological self-closeness number \cite[Theorem 41]{OY20}.

	\begin{pro}
		\label{upper bound cohomology}
		Let $X$ be a path-connected space and of finite type. Then
		\[
		N\mathcal{E}(X)\le N^*\mathcal{E}(X).
		\]
	\end{pro}
	
	\begin{proof}
		Suppose that a map $f\colon X\to X$ is an isomorphism in homotopy group up to dimension $n$.

		By the Hurewicz theorem, the $f$ is an isomorphism in $H_*(X)$ for $* <n$ and surjective in $H_n(X)$. On the other hand, a surjective self-map of a finitely generated abelian group is an isomorphism. Since $X$ is of finite type, $H_n(X)$ is a finitely generated abelian group, so $f$ is an isomorphism in $H_n(X)$. Then we get that $f$ is an isomorphism in $H_*(X)$ for $*\le n$.

		By the universal coefficient theorem, there is a commutative diagram
		\[
		\xymatrix{
			0\ar[r]&Ext(H_{i-1}(X),\mathbb{Z})\ar[r]\ar[d]^{(f_*)^*}&H^i(X)\ar[r]\ar[d]^{f^*}&Hom(H_i(X),\mathbb{Z})\ar[r]\ar[d]^{(f_*)^*}&0\\
			0\ar[r]&Ext(H_{i-1}(X),\mathbb{Z})\ar[r]&H^i(X)\ar[r]&Hom(H_i(X),\mathbb{Z})\ar[r]&0
		}
		\]
		with exact rows for $i\le n$. Thus by the five lemma, $f$ is an isomorphism in cohomology up to dimension $n$, completing the proof.
	\end{proof}

	\begin{exa}
		It is easy to give a space for which the equality holds in Proposition \ref{upper bound cohomology}. On the other hand, we can give an example of a space for which the inequality is strict such as
		\[
		N\mathcal{E}(S^2\cup_2e^3)=2\quad\text{and}\quad N^*\mathcal{E}(S^2\cup_2e^3)=3.
		\]
		We can see, in the special case, how big the difference $N^*\mathcal{E}(X)-N\mathcal{E}(X)$ is. It is proved in \cite{OY20} that for a CW complex $X$ of finite type, we have $N^*\mathcal{E}(X)\le N_*\mathcal{E}(X)+1$. Thus by \eqref{homotopy vs homology}, if $X$ is simply-connected additionally, then
		\[
		N^*\mathcal{E}(X)\le N\mathcal{E}(X)+1
		\]
	\end{exa}

	We further give a computable upper bound for $N^*\mathcal{E}(X)$. 
	\begin{lem}
		\label{cohomology generator}
		Let $X$ be a simply-connected space. Then
		\[
		N^*\mathcal{E}(X)\le d(H^*(X;\Z))
		\]
	\end{lem}
	
	\begin{proof}
		Since $d(H^*(X;\Z))$ is identified with the maximal degree of generators of $H^*(X;\Z)$, the statement follows.
	\end{proof}
	\section{$p$-Universality}
	This section recalls $p$-universality introduced by Mimura, O'Neill and Toda \cite{MOT71}. We also give important properties and examples of $p$-universal spaces. For $p$ prime or $p=0$, a $p$-equivalence is a map between spaces inducing an isomorphism on mod $p$ homology (rational homology for $p=0$). By \cite[Chapter 2, Theorem 1.14]{HMR75Ch2}, a $p$-equivalence between simply-connected CW complexes of finite type is equivalent to a $p$-local homotopy equivalence between such spaces. 
	
	\noindent
	Notice that a homotopy equivalence $f:X\to Y$ induces bijection between homotopy sets
	\begin{equation}
		\label{f_*}
		f_*:[K,X]\to[K,Y]
	\end{equation}
	for any space $K$. However, $f_*$ is no longer a bijection if $f$ is an $p$-equivalence instead. This leads us to the following definition, with which \eqref{f_*} is surjective in a weak sense. 
	\begin{dfn}
		\label{p-uni}
		A simply-connected finite CW complex $K$ is $p$-universal if for any $p$-equivalence $f:X\to Y$ and any map $g:K\to Y$, there is a diagram commutes up to homotopy 
		\[
		\xymatrix{
			X\ar[r]^f&Y\\
			K\ar@{.>}[u]\ar@{.>}[r]^r&K\ar[u]_g
		}
		\]
		where $r$ is a $p$-equivalence. 
		
	\end{dfn}
	Mimura, O'Neill and Toda \cite{MOT71} also gave several equivalent conditions for $p$-universality, and the following one is useful among those: 
	\begin{lem}
		\label{2equi}
		A simply-connected finite CW complex $K$ is $p$-universal if and only if 
		for any finite abelian group $G$ satisfying $p\nmid |G|$, there exists a $p$-equivalence $f:K\to K$ such that $f^*$ is trivial on $H^*(K;G)$. 
	\end{lem}
	
	Definition \ref{p-uni} did not demand any weak injectivity of the induced map \eqref{f_*}, and Mimura, O'Neill and Toda did not consider it in \cite{MOT71} either. However, we can deduce that \eqref{f_*} is injective in a weak sense as follows.  
	\begin{pro}
		\label{inj}
		Let $K$ be a $p$-universal space and $f$ be a $p$-equivalence between simply-connected CW complexes of finite type in \eqref{f_*}, then for any maps $h_1,h_2:K\to X$ satisfying $f\circ h_1\simeq f\circ h_2$, there exists a $p$-equivalence $r:K\to K$ such that 
		\[h_1\circ r\simeq h_2\circ r\]
	\end{pro}
	\begin{proof}
		For maps $h_1,h_2:K\to X$ which $f\circ h_1\simeq f\circ h_2$, the obstruction of the homotopy between $h_1$ and $h_2$ lies in $H^{n+1}(K\times I,K\times\partial I;\pi_n(Y,X))$. 
		\[
		\xymatrix{
			K\times I\ar[r]\ar@{.>}[dr]&Y\\
			K\times\partial I\ar[r]_{\quad \{h_1,h_2\}}\ar[u]&X\ar[u]_f
		}
		\]	
		Since $f$ is a $p$-equivalence between simply-connected CW complexes of finite type, it is also a $p$-local homotopy equivalence, hence $\pi_n(Y,X)\in\mathcal{C}_p$ for all $n$. By Lemma \ref{2equi}, there exists a $p$-equivalence $r:K\to K$ inducing a trivial map on 
		\begin{align*}
			&H^{n+1}(K\times I,K\times\partial I;\pi_n(Y,X)) \\
			&\cong H^{n+1}(K\ltimes S^1;\pi_n(Y,X)) \\
			&\cong H^{n+1}(S^1;\pi_n(Y,X))\oplus H^{n+1}(\Sigma K;\pi_n(Y,X))
		\end{align*}
		completing the proof. 
	\end{proof}
	
	\begin{rem}
		\label{f*}
		A homotopy equivalence $f$ also induces bijection 
		\begin{equation}
			\label{f^*}
			f^*:[Y,K]\to[X,K]
		\end{equation}
		for any space $K$. Mimura, O'Neill and Toda also proved that the map \eqref{f^*} is surjective in a weak sense if and only if $K$ is $p$-universal and $f$ is a $p$-equivalence between simply-connected finite complexes \cite[Theorem 2.1]{MOT71}. Namely, for any map $g:X\to K$, there exists a $p$-equivalence $r:K\to K$ and a map $h:Y\to K$ such that the following diagram commutes up to homotopy: 
		\[
		\xymatrix{
			X\ar[r]^f\ar[d]^g&Y\ar@{.>}[d]^h\\
			K\ar@{.>}[r]^r&K
		}
		\]
		As well as Proposition \ref{inj}, \eqref{f^*} is also injective in a weak sense. 
	\end{rem}

	In general if a map $X\to Y$ is a $p$-equivalence, it needs not exist a $p$-equivalence from $Y$ to $X$. However, the $p$-universality guarantees the existence of the inverse map, that is, for those involving a $p$-universal space we have: 
	\begin{lem}
		\label{inverse}
		Let $f:X\to Y$ be a $p$-equivalence between simply-connected finite complexes. If either $X$ or $Y$ is $p$-universal, then there exists a $p$-equivalence $h:Y\to X$. 
	\end{lem}
	\begin{proof}
		If $Y$ is $p$-universal, let $K=Y$ and $g=id_Y$ in the diagram in Definition \ref{p-uni}
		\[
		\xymatrix{
			X\ar[r]^f&Y\\
			Y\ar@{.>}[u]^h\ar[r]^r&Y\ar@{=}[u]
		}
		\]
		then there exist map $h:Y\to X$ and $p$-equivalence $r:Y\to Y$ such that $f\circ h\simeq r$, hence $g$ is also a $p$-equivalence. 
		
		\noindent
		For $X$ $p$-universal, pluge $K=X$ and $g=id_X$ into the diagram in Remark \ref{f*} 
		\[
		\xymatrix{
			X\ar[r]^f\ar@{=}[d]&Y\ar@{.>}[d]^h\\
			X\ar[r]^r&X
		}
		\]
		then we also get an inverse $p$-equivalence $h$. 
	\end{proof}

	Next we introduce $mod\ p$ $H$-spaces and $mod\ p$ $co$-$H$-spaces which are crucial examples of $p$-universal spaces, and we will study their self-closeness numbers in the next two sections. Recall that a simply-connected space is a $mod\ p$ $(co$-)$H$-space if its $p$-localization is an $(co$-$)H$-space. We say that a $mod\ p$ $(co$-$)H$-space is finite if it is a finite CW complex.

	Mimura, O'Neill and Toda \cite[Theorem 4.2]{MOT71} proved the following proposition: 
	\begin{pro}
		\label{modp uni}
		Any finite $mod\ p$ $(co$-$)H$-space is $p$-universal. 
	\end{pro}
	In the rest of this paper, $mod\ 0$ $(co$-$)H$-space will be denoted as $(co$-$)H_0$-space instead. There is a well-known characterization for $H_0$-space \cite{MT71}. 
	\begin{pro}
		\label{H_0}
		For a simply-connected finite CW complex $X$, the following conditions are equivalent: 
		\begin{enumerate}
			\item [(a)] 
			$X$ is an $H_0$-space. 
			\item [(b)] 
			There exists a 0-equivalence $\prod_{i=1}^r S^{n_i}\to X$ for odd integers $n_i$ satisfying $2\le n_1\le\dots\le n_r$. 
			\item [(c)]
			All $k$-invariants of $X$ is with finite order. 
		\end{enumerate}
		The sequence $n_1,\dots,n_r$ is called the {\emph type} of an $H_0$-space $X$. 
	\end{pro}
	
	Similar result for co-H$_0$-spaces was proved by Arkowitz and Curjel \cite[Theorem 2.5]{AC64}. 
	\begin{pro}
		\label{co-H_0}
		For a simply-connected finite CW complex $X$, the following conditions are equivalent: 
		\begin{enumerate}
			\item [(a)] 
			$X$ is an $co$-$H_0$-space. 
			\item [(b)] 
			There exists a 0-equivalence $X\to\bigvee_{i=1}^r S^{n_i}$ for $2\le n_1\le\dots\le n_r$. 
		\end{enumerate}
		The sequence $n_1,\dots,n_r$ is called the {\emph type} of an $co$-$H_0$-space $X$. 
	\end{pro}
	We will use the following properties of $(co$-$)H_0$-spaces. 
	\begin{pro}
		\label{degree}
		Let $X$ be a finite $(co$-$H)$-space of type $n_1,\dots,n_r$, then for each $i=1,\dots,r$, there are maps 
		\[f_i:S^{n_i}\to X\ \text{and}\ g_i:X\to S^{n_i}\]
		such that $g_i\circ f_i$ is an 0-equivalence. 
	\end{pro}
	\begin{proof}
		We will prove the proposition for $H_0$-spaces and the proof for $co$-$H_0$-space is similar. By Proposition \ref{modp uni}, $X$ is 0-universal. So by Lemma \ref{inverse} and Proposition \ref{H_0}, there are 0-equivalences 
		\[f:\prod_{i=1}^r S^{n_i}\to X\ \text{and}\ g:X\to\prod_{i=1}^r S^{n_i}\]
		Since both maps are rational homotopy equivalence (cf. \cite[Theorem 1.14]{HMR75Ch2}), they are completely determined by matrixes \[F,G\in GL_r(\mathbb{Q})\]
		So we obtain an inverse matrix $H'\in GL_r(\mathbb{Q})$ of $GF$. Notice that there exists a diagonal matrix $A=diag(m_1,\dots,m_r)$ such that 
		\[H=AH'\in M_r(\mathbb{Z})\]
		which is invertible and determines a 0-equivalence 
		\[h:\prod_{i=1}^r S^{n_i}\to\prod_{i=1}^r S^{n_i}\]
		Then $h\circ g\circ f:\prod_{i=1}^r S^{n_i}\to\prod_{i=1}^r S^{n_i}$  is determined by $HGF=AH'GF=A$, which is a product of 0-equivalences on $S^{n_i}$. Thus $f_i$ and $g_i$ are obtained by compositions
		\[S^{n_i}\to \prod_{i=1}^r S^{n_i}\stackrel{h\circ f}{\longrightarrow}X\ \text{and}\ X\stackrel{g}{\longrightarrow}\prod_{i=1}^r S^{n_i}\to S^{n_i}\] 
		and $g_i\circ f_i$ is a degree $m_i$ map, hence a 0-equivalence. 
	\end{proof}

	\section{Proof of Theorem \ref{main} for ${H}_0$-spaces}
	
	First of all, we set notation that we are going to use throughout this and next section. Let $G$ be a finite group, and let $X$ be a simply-connected finite $G$-complex. Let $q\colon X\to X/G$ denote the projection. We consider a condition on the action of $G$ on $X$.
	
	\begin{lem}
		\label{action}
		The following two statements are equivalent:
		\begin{enumerate}
			\item [(a)] 
			the map $q^*\colon H^*(X/G;\Q)\to H^*(X;\Q)$ is an isomorphism;
			
			\item [(b)] 
			the action of $G$ on $H^*(X;\Q)$ is trivial;
		\end{enumerate}
	\end{lem}
	
	\begin{proof}
		By \cite[Proposition 3G.1]{Hatcher02}, $H^*(X/G;\Q)$ is isomorphic to the invariant ring $H^*(X;\Q)^G$ such that the map $q^*\colon H^*(X/G;\Q)\to H^*(X;\Q)$ is identified with the inclusion $H^*(X;\Q)^G\to H^*(X;\Q)$. Then (a) and (b) are equivalent.
	\end{proof}

	Hereafter, we assume that the action of $G$ on $X$ satisfies either of the equivalent conditions in Lemma \ref{action}.

	In particular, suppose that $X$ is an $H_0$-space in the rest of this section. We further assume that for some $n\ge 2$, there are maps $\epsilon\colon S^{2n-1}\to X$ and $\rho\colon X\to S^{2n-1}$ such that $\rho\circ\epsilon\colon S^{2n-1}\to S^{2n-1}$ is a 0-equivalence. If $H^{2n-1}(X;\Q)\ne 0$, then these maps do exist. By composing with a self-map of $S^{2n-1}$ of degree -1 if necessary, we may assume $\rho\circ\epsilon\colon S^{2n-1}\to S^{2n-1}$ is of positive degree $M$.

	Let $Y^m$ denote the $m$-skeleton of a CW complex $Y$.
	\begin{lem}
		\label{mu}
		For some positive integer $N$, there is a homotopy commutative diagram
		\[
		\xymatrix{
			X\vee S^{2n-1}\ar[r]^(.65){1+N\epsilon}\ar[d]&X\ar@{=}[d]\\
			X\times S^{2n-1}\ar[r]^(.55)\mu&X.
		}
		\]
	\end{lem}
	
	\begin{proof}
		Suppose a map $i_m+N_m\epsilon\colon X^{m-1}\vee S^{2n-1}\to X$ extends over $X^{m-1}\times S^{2n-1}$ for a positive integer $N_m$, where $i_m:X^m\to X$ is the inclusion. Consider the obstruction class
		\[
		\mathfrak{o}(a)\in H^{m+1}(X\times S^{2n-1},X\vee S^{2n-1};\pi_m(X))
		\]
		for extending a map $1+aN_m\epsilon\colon X^m\vee S^{2n-1}\to X$ over $X^m\times S^{2n-1}$, where $a\in\Z$. There is an isomorphism
		\[
		H^{m+1}(X\times S^{2n-1},X\vee S^{2n-1};\pi_m(X))\cong H^{m+1}(X\wedge S^{2n-1};\pi_m(X))
		\]
		which is natural with respect to self-maps of $S^{2n-1}$. Then in particular, $\mathfrak{o}(a)=a\mathfrak{o}(1)$. On the other hand, $\mathfrak{o}(1)$ factors through the $k$-invariants $k_m$ of $X$. Since $X$ is an ${H}_0$-space,  $k_m$ is of finite order and so is $\mathfrak{o}(1)$. Thus since $X$ is finite dimensional, the proof is complete.
	\end{proof}

	Define a map $f\colon X\to X$ by the composition
	\[
	X\xrightarrow{\Delta}X\times X\xrightarrow{1\times\rho}X\times S^{2n-1}\xrightarrow{\mu}X
	\]
	where $\Delta$ is the diagonal map and $\mu$ is as in Lemma \ref{mu}. Clearly, $f$ is an isomorphism in $\pi_*$ for $*<2n-1$. On the other hand, since $\rho\circ\epsilon\colon S^{2n-1}\to S^{2n-1}$ is of positive degree, we have $f_*(\epsilon)=(MN+1)\epsilon\in\pi_{2n-1}(X)$. Then $f$ is not an isomorphism in $\pi_{2n-1}$ because $\epsilon\in\pi_{2n-1}(X)$ is of infinite order. Thus
	\[
	N\mathcal{E}(X)\ge 2n-1
	\]
	We will get a lower bound for $N\mathcal{E}(X/G)$ basically the same argument. Then we need to construct maps
	\[
	\bar{\rho}\colon X/G\to S^{2n-1}\quad\text{and}\quad\bar{\mu}\colon X/G\times S^{2n-1}\to X/G
	\]
	having properties analogous to $\rho\colon X\to S^{2n-1}$ and $\mu\colon X\times S^{2n-1}\to X$.

	First, we construct a map $\bar{\rho}\colon X/G\to S^{2n-1}$. For an integer $k$, let $\underline{k}\colon S^{2n-1}\to S^{2n-1}$ denote a map of degree $k$. Then for any $\alpha\in\pi_{i}(S^{2n-1})$, we can consider $\underline{k}\circ\alpha$, which possibly differs from $k\alpha$. However, we have:

	\begin{lem}
		\label{Hilton-Milnor}
		For any $\alpha\in\pi_i(S^{2n-1})$ with $i>2n-1$, there is a positive integer $k$ such that $\underline{k}\circ\alpha=0$.
	\end{lem}
	
	\begin{proof}
		Let $\alpha\in\pi_i(S^{2n-1})$ with $i>2n-1$. Then $\alpha$ is of finite order. On the other hand, by \cite[Theorem 6.7]{Hilton55}, $\underline{m}\circ\alpha=m\alpha$ whenever $m$ is divisible by 4. Then the proof is done by setting $k=4\cdot\mathrm{ord}(\alpha)$.
	\end{proof}

	\begin{pro}
		\label{rho X/G}
		For some positive integer $k$, there is a homotopy commutative diagram
		\[
		\xymatrix{
			X\ar[r]^{\rho}\ar[d]_q&S^{2n-1}\ar[d]^{\underline{k}}\\
			X/G\ar[r]^{\bar{\rho}}&S^{2n-1}.
		}
		\]
	\end{pro}
	
	\begin{proof}
		We induct on the skeleta of CW pairs $(M_q,X)$, where $M_q$ is the mapping cylinder of $q$ and has the homotopy type of $X/G$. Clearly, the induction begins with the $X\cup M_q^{2n-1}$. Consider the obstruction class
		\[
		\mathfrak{o}^{2n}(a)\in H^{2n}(C_q;\pi_{2n-1}(S^{2n-1}))
		\]
		for extending $(\underline{a}\circ\rho)_{2n-1}\colon X\cup M_q^{2n-1}\to S^{2n-1}$ over $X\cup M_q^{2n}$, where $(\underline{a}\circ\rho)_{2n-1}$ extends $\underline{a}\circ\rho$ and $C_q$ denotes the mapping cone of $q\colon X\to X/G$. Since $\pi_{2n-1}(S^{2n-1})\cong\Z$,
		\[
		H^{2n}(C_q;\pi_{2n-1}(S^{2n-1}))\otimes\Q\cong H^{2n}(C_q;\Q)
		\]
		and by Lemma \ref{action}, $H^{2n}(C_q;\Q)=0$. Then the obstruction class $\mathfrak{o}^{2n}(1)$ is of finite order. On the other hand, by naturality of the obstruction class, we have
		\[
		\mathfrak{o}^{2n}(a)=a\mathfrak{o}(1).
		\]
		Thus $\mathfrak{o}(k_{2n})=0$ for some positive integer $k_{2n}$, implying the existence of $(\underline{k_{2n}}\circ\rho)_{2n}\colon X\cup M_q^{2n}\to S^{2n-1}$ which extends $\underline{k_{2n}}\circ\rho$.

		Now we suppose that $(\underline{k_m}\circ\rho)_m\colon X\cup M_q^{m}\to S^{2n-1}$ extends $\underline{k_m}\circ\rho$ for some positive integer $k_m$, where $m>2n-1$. Consider the obstruction class
		\[
		\mathfrak{o}^{m+1}(b)\in H^{m+1}(C_q;\pi_m(S^{2n-1}))
		\]
		for extending $(\underline{b}\circ\underline{k_m}\circ\rho)_m\colon X\cup M_q^{m}\to S^{2n-1}$ over $X\cup M_q^{m+1}$. Since $\pi_m(S^{2n-1})$ is a finite abelian group, it follows from Lemma \ref{Hilton-Milnor} that there is a positive integer $k_{m+1}$ such that $\underline{k_{m+1}}\circ\alpha=0$ for each $\alpha\in\pi_m(S^{2n-1})$, implying $\mathfrak{o}^{m+1}(k_{m+1})=0$. Then $(\underline{k_{m+1}}\circ\underline{k_m}\circ\rho)_{m+1}\colon X\cup M_q^{2n-1}\to S^{2n-1}$ extends $\underline{k_{m+1}}\circ\underline{k_m}\circ\rho$. Thus since $M_q$ is of finite dimension, the proof is complete.
	\end{proof}

	Next, we construct a map $\bar{\mu}\colon X/G\times S^{2n-1}\to X/G$ by extending a map $1+l(q\circ\epsilon)\colon X/G\vee S^{2n-1}\to X/G$ for some integer $l$. To this end, we will perform obstruction theory, so we show a property of the obstruction class. Suppose that the map $i_m+l_m(q\circ\epsilon)\colon(X/G)^m\vee S^{2n-1}\to X/G$ over $(X/G)^m\times S^{2n-1}$, where $i_m\colon(X/G)^m\to X/G$ denotes the inclusion. Then we can define the obstruction class \cite{Baues77Ch4}
	\begin{equation}
		\label{obstruction}
		\mathfrak{o}^{m+1}(a)\in H^{m+1}(X/G\times S^{2n-1},X/G\vee S^{2n-1};\underline{\pi_m(X/G)})
	\end{equation}
	for extending $i_{m+1}+al_m(q\circ\epsilon)\colon(X/G)^{m+1}\vee S^{2n-1}\to X/G$ over $(X/G)^{m+1}\times S^{2n-1}$, where $\underline{\pi_*(X/G)}$ denotes the twisted coefficients associated to the canonical action of $G$ on $\pi_*(X/G)$.
	
	\begin{lem}
		\label{Kunneth}
		There is an isomorphism
		\[
		H^{m+1}(X/G\times S^{2n-1},X/G\vee S^{2n-1};\underline{\pi_m(X/G)})\cong H^{m-2n+2}(X/G,*;\underline{\pi_m(X/G)})\otimes H^{2n-1}(S^{2n-1})
		\]
		which is natural with respect to self-maps of $S^{2n-1}$.
	\end{lem}
	
	\begin{proof}
		Let $C_*(Y)$ denotes the cellular chain complex of a CW complex $Y$. Let $W=X/G\vee S^{2n-1}$, and let $\widetilde{Y}$ denotes the universal covering space of space $Y$. Since $G$ acts freely on $X$, $C_*(X\times S^{2n-1})$ and $C_*(\widetilde{W})$ are free $\Z G$-modules. By definition, $H^{*+1}(X/G\times S^{2n-1},X/G\vee S^{2n-1};\underline{\pi_*(X/G)})$ is the cohomology of
		\[
		\mathrm{Hom}_G(C_{*+1}(X\times S^{2n+1})/C_{*+1}(\widetilde{W}),\underline{\pi_*(X/G)})
		\]
		where $\mathrm{Hom}_G(A,B)$ denotes the abelian group of $G$-equivariant maps between $\Z G$-modules $A$ and $B$. Then we consider the $\Z G$-module $C_{*+1}(X\times S^{2n+1})/C_{*+1}(\widetilde{W})$. Notice that $\widetilde{W}$ is obtained by attaching $S^{2n-1}$ to $X$ at each point of $q^{-1}(*)$. Then
		\[
		C_{*+1}(\widetilde{W})=C_{*+1}(X)\oplus(C_0(q^{-1}(*))\otimes C_{2n-1}(S^{2n-1}))
		\]
		and so we get an isomorphism of $\Z G$-modules
		\[
		C_{*+1}(X\times S^{2n+1})/C_{*+1}(\widetilde{W})\cong C_{*-2n+2}(X,q^{-1}(*))\otimes C_{2n-1}(S^{2n-1}).
		\]
		Thus the proof is complete.
	\end{proof}

	The following is immediate from Lemma \ref{Kunneth}.
	
	\begin{cor}
		\label{linearity}
		The obstruction class \eqref{obstruction} satisfies
		\[
		\mathfrak{o}^{m+1}(a)=a\mathfrak{o}^{m+1}(1).
		\]
	\end{cor}

	Now we prove:
	
	\begin{pro}
		\label{mu X/G}
		Let $2n-1=d(H^*(X;\Q))$. Then for some positive integer $l$, there is a homotopy commutative diagram
		\[
		\xymatrix{
			X/G\vee S^{2n-1}\ar[r]^(.62){1+l(q\circ\epsilon)}\ar[d]&X/G\ar@{=}[d]\\
			X/G\times S^{2n-1}\ar[r]^(.6){\bar{\mu}}&X/G.
		}
		\]
	\end{pro}
	
	\begin{proof}
		Since $X_{(0)}$ is a product of rationalized odd spheres of dimensions $\le 2n-1$. In particular, $\pi_*(X/G)$ is a finite abelian group for $*>2n-1$.
		
		Clearly, $i_{2n-1}+q\circ\epsilon\colon(X/G)^{2n-1}\vee S^{2n-1}\to X/G$ extends over $X/G$. Suppose $i_m+l_m(q\circ\epsilon)\colon(X/G)^{m}\vee S^{2n-1}\to X/G$ extends over $X/G$. Then we can define the obstruction class \eqref{obstruction}, where $m>2n-1$. Since $\pi_m(X/G)$ is a finite abelian group, $l_{m+1}\mathfrak{o}^{m+1}(1)=0$ for some positive integer $l_{m+1}$. Then by Lemma \ref{linearity}, $\mathfrak{o}^{m+1}(l_{m+1})=0$, implying 
		$$i_{m+1}+l_{m+1}l_m(q\circ\epsilon)\colon(X/G)^{m+1}\vee S^{2n-1}\to X/G$$
		extends over $X/G$. Since $X/G$ is finite dimensional, the proof is done.
	\end{proof}

	\begin{thm}
		\label{lower bound H_0}
		There is an inequality
		\[
		N\mathcal{E}(X/G)\ge d(X;\Q).
		\]
	\end{thm}
	
	\begin{proof}
		Let $\bar{\rho}$ and $\bar{\mu}$ be as in Propositions \ref{rho X/G} and \ref{mu X/G}, where we set $2n-1=d(H^*(X;\Q))$. We define a map $\bar{f}\colon X/G\to X/G$ by the composition
		\[
		X/G\xrightarrow{\Delta_{X/G}}X/G\times X/G\xrightarrow{1\times\bar{\rho}}X/G\times S^{2n-1}\xrightarrow{\bar{\mu}}X/G
		\]
		where $\Delta_Y\colon Y\to Y\times Y$ is the diagonal map. Since $\bar{\mu}\vert_{X/G}\times*=1_{X/G}$, $\bar{f}$ is an isomorphism in $\pi_*$ for $*<2n-1$. On the other hand,
		\begin{align*}
			\bar{f}_*(q\circ\epsilon)&=\bar{\mu}\circ(1\times\bar{\rho})\circ\Delta_{X/G}\circ q\circ\epsilon\\
			&=\bar{\mu}\circ(1\times\bar{\rho})\circ((q\circ\epsilon)\times(q\circ\epsilon))\circ\Delta_{S^{2n-1}}\\
			&=\bar{\mu}\vert_{X/G\vee S^{2n-1}}\circ(1\vee\bar{\rho})\circ((q\circ\epsilon)\vee(q\circ\epsilon))\circ\psi\\
			&=(1+l(q\circ\epsilon))\circ((q\circ\epsilon)\vee(\underline{k}\circ\rho\circ\epsilon))\circ\psi\\
			&=((q\circ\epsilon)+lkM(q\circ\epsilon))\circ\psi\\
			&=(1+klM)(q\circ\epsilon)
		\end{align*}
		where $\psi\colon S^{2n-1}\to S^{2n-1}\vee S^{2n-1}$ is the pinch map. Thus since $l,k,M$ are positive integers and $q\circ\epsilon\in\pi_{2n-1}(X/G)$ is of infinite order, $\bar{f}$ is not an isomorphism in $\pi_{2n-1}$.
	\end{proof}

	\begin{proof}
		[Proof of Theorem \ref{main} for ${H}_0$-spaces]
		Combine Lemma \ref{cohomology generator} and Theorem \ref{lower bound H_0}.
	\end{proof}
	\section{Proof of Theorem \ref{main co-H} for co-${H}_0$-spaces}
	
	We set similar notations and conditions for space $X$ and group $G$ in this section. Let $G$ be a finite group, and let $X$ be a simply-connected finite $G$-complex. Let $q\colon X\to X/G$ denote the projection. We consider a condition on the action of $G$ on $X$. We also assume that the action of $G$ on $X$ satisfies either of the equivalent conditions in Lemma \ref{action}. 
	
	Recall that the definition of the cohomological dimension of a path-connected space $Y$ is 
	\[cd(Y)=\sup\{d\mid H^d(Y;M)\ne 0\text{ for some }\Z\pi_1(Y)\text{-module }Y\}\]
	\begin{lem}
		\label{cd}
		Suppose $cd(X/G)=d(H^*(X/G;\Q))=n$ holds for $X$ and $G$ above, then 
		$$cd(X)=d(H^*(X;\Z))=n$$
	\end{lem}
	\begin{proof}
		For any path-connected space $Y$, notice that let $M=A\pi_1(Y)$ for an abelian group $A$, then $H^d(Y;M)\cong H^d(\widetilde{Y};A)$ where $\widetilde{Y}$ is the universal covering space of $Y$. Hence we have $cd(Y)\ge cd(\widetilde{Y})$.

		By the canonical inequality $cd(Y)\ge d(H^*(Y;\Z))\ge d(H^*(Y;\Q))$, let $Y=X/G$ and we have 
		\[cd(X/G)\ge cd(X)\ge d(H^*(X;\Z))\ge d(H^*(X;\Q))=d(H^*(X/G;\Q))\] 
		where the last equal sign is by Lemma \ref{action}. 
		So since $cd(X/G)=d(H^*(X/G;\Q))$, all "$\ge$" turn to "$=$" and we obtain that 
		\begin{equation}
			\label{coH equation}
			cd(X/G)=cd(X)=d(H^*(X;\Z))=d(H^*(X;\Q))=n
		\end{equation}
	\end{proof}
	In the rest of this section, we assume that $X$ is an co-$H_0$-space. By Proposition \ref{degree}, there are maps $\epsilon\colon S^{n}\to X$ and $\rho\colon X\to S^{n}$ such that $\rho\circ\epsilon\colon S^{n}\to S^{n}$ is a 0-equivalence. By composing with a self-map of $S^{n}$ of degree -1 if necessary, we may assume $\rho\circ\epsilon\colon S^{n}\to S^{n}$ is of positive degree $M$.

	\begin{lem}
		\label{phi}
		There is a homotopy commutative diagram
		\[
		\xymatrix{
			X\ar[r]^(.35){\phi}\ar@{=}[d]&X\vee S^n\ar[d]\\
			X\ar[r]^(.4){1\times\rho}&X\times S^n
		}
		\]
	\end{lem}
	
	\begin{proof}
		Let $i:X\vee S^n\to X\times S^n$ be the inclusion and $F_i$ be its homotopy fiber. Since $X\vee S^n$ is simply-connected, the map has a Moore-Postnikov decomposition 
		\begin{equation}
			\xymatrix{
				&Z^{(m+1)}\ar[d]^{p_{m+1}}\ar[dr]& \\
				X\vee S^n\ar[r]\ar[ur]&Z^{(m)}\ar[r]&X\times S^n 
			}
		\end{equation}
		where the vertical map is a principal fibration induced by $k^{m+1}:Z^{(m)}\to K(\pi_m(F_i),m+1)$.

		Suppose there is a map $(1\times\rho)_m\colon X\to Z^{(m)}$ lifting $1\times\rho:X\to X\times S^n$ for a positive integer $N$, and consider the obstruction class
		\[
		\mathfrak{o}\in H^{m+1}(X;\pi_m(F_i))
		\]
		for lifting $(1\times\rho)_m\colon X\to Z^{(m)}$ to $Z^{(m+1)}$ through $p_{m+1}$. Since $i$ is a $n$-homological equivalence and both $X\vee S^n$, $X\times S^n$ are simply-connected, $i$ is a $n$-equivalence by Whitehead theorem, hence $F_i$ is $(n-1)$-connected. Since $cd(X)=n$, we conclude that 
		\[H^{m+1}(X;\pi_m(F_i))=0\]
		for any integer $m\ge0$, completing the proof. 
	\end{proof}

	Define a map $g\colon X\to X$ by the composition
	\[
	X\xrightarrow{\phi}X\vee S^n\xrightarrow{1+\epsilon}X\vee X\xrightarrow{\nabla}X
	\]
	where $\nabla$ is the folding map and $\phi$ is as in Lemma \ref{phi}. Clearly, $q$ is an isomorphism in $\pi_*$ for $*<n$. On the other hand, since $\rho\circ\epsilon\colon S^{n}\to S^{n}$ is of positive degree, we have $g_*(\epsilon)=(MN+1)\epsilon\in\pi_{n}(X)$. Then $f$ is not an isomorphism in $\pi_{n}$ because $\epsilon\in\pi_{n}(X)$ is of infinite order. Thus
	\[
	N\mathcal{E}(X)\ge n
	\]
	We will get a lower bound for $N\mathcal{E}(X/G)$ basically the same argument. Then we need to construct maps
	\[
	\bar{\rho}\colon X/G\to S^{n}\quad\text{and}\quad\bar{\phi}\colon X/G\to X/G\vee S^n
	\]
	having properties analogous to $\rho\colon X\to S^{n}$ and $\phi\colon X\to X\vee S^n$.

	First, we construct a map $\bar{\rho}\colon X/G\to S^{n}$. Since $n$ may be even number here, we cannot apply Lemma \ref{Hilton-Milnor}. However, since the cohomological dimension of $X$ is given, we will still be able to prove the existence of $\bar{\rho}$. 
	
	\begin{pro}
		\label{rho X/G2}
		For some positive integer $k$, there is a homotopy commutative diagram
		\[
		\xymatrix{
			X\ar[r]^{\rho}\ar[d]_q&S^{n}\ar[d]^{\underline{k}}\\
			X/G\ar[r]^{\bar{\rho}}&S^{n}
		}
		\]
	\end{pro}
	
	\begin{proof}
		We induct on the skeleta of CW pairs $(M_q,X)$, where $M_q$ is the mapping cylinder of $q$ and has the homotopy type of $X/G$. Clearly, the induction begins with the $X\cup M_q^n$. Suppose there is an extension $\rho_n:X\cup M_q^n\to S^{n}$ of $\rho$. Consider the obstruction class
		\[
		\mathfrak{o}(k)\in H^{n+1}(C_q;\pi_{n}(S^{n}))
		\]
		for extending $\underline{k}\circ\rho_n\colon X\cup M_q^n\to S^{n}$ over $X\cup M_q^{n+1}$, where $C_q$ denotes the mapping cone of $q\colon X\to X/G$. By the cohomology long exact sequence, $cd(X/G)=cd(X)=n$ implies $cd(C_q)\le n+1$.

		Since $\pi_{n}(S^{n})\cong\Z$,
		\[
		H^{n+1}(C_q;\pi_{n}(S^{n}))\otimes\Q\cong H^{n+1}(C_q;\Q)
		\]
		and by the assumption on the action of $G$ on $X$, $H^{n+1}(C_q;\Q)=0$. Then the obstruction class $\mathfrak{o}(1)$ is of finite order. On the other hand, by naturality of the obstruction class, we have
		\[
		\mathfrak{o}(k)=k\mathfrak{o}(1).
		\]
		Thus $\mathfrak{o}(k)=0$ for some positive integer $k$, implying that there is a map $(\underline{k}\circ\rho)_{n+1}:X\cup M_q^{n+1}\to S^{n}$ extending $\underline{k}\circ\rho$.

		Since $H^{m+1}(C_q;\pi_m(S^n))=0$ for any $m>n$, there is no obstruction for extending $(\underline{k}\circ\rho)_{n+1}:X\cup M_q^{n+1}\to S^{n}$ over $M_q\simeq X/G$ and the proof is complete. 
	\end{proof}

	Next, we construct a map $\bar{\phi}\colon X/G\to X/G\vee S^n$ by lifting the map $1\times \bar{\rho}\colon X/G\to X/G\times S^n$ through the inclusion $i':X/G\vee S^n\to X/G\times S^n$. To this end, we will perform obstruction theory, so we show a property of the obstruction class. Similar to Lemma \ref{phi}, we consider the Moore-Postnikov decomposition of $i':X/G\vee S^n\to X/G\times S^n$ 
	\begin{equation}
		\xymatrix{
			&W^{(m+1)}\ar[d]^{p'_{m+1}}\ar[dr]& \\
			X/G\vee S^n\ar[r]\ar[ur]&W^{(m)}\ar[r]&X/G\times S^n 
		}
	\end{equation}
	Suppose that there is a map $(1\times\bar{\rho})_m:X/G\to W^{(m)}$ lifting $1\times\bar{\rho}\colon X/G\to X/G\times S^n$, then we can define the obstruction class
	\begin{equation}
		\label{coobstruction}
		\mathfrak{o}\in H^{m+1}(X/G;\underline{\pi_m(F_{i'})})
	\end{equation}
	for lifting $(1\times\bar{\rho})_m:X/G\to W^{(m)}$ through $p'_{m+1}$, where $F_{i'}$ to the homotopy fiber of $i':X/G\vee S^n\to X/G\times S^n$ and $\underline{\pi_m(F_{i'})}$ denotes the twisted coefficients associated to the action of $\pi_1(X/G\vee S^n)\cong G$ on $\pi_m(F_{i'})\cong\pi_{m+1}(X/G\times S^n,X/G\vee S^n)$.

	In order to analyze the obstruction class \eqref{coobstruction}, we need to compute the connectivity of $F_{i'}$.  
	
	\begin{lem}
		\label{ss}
		The homotopy fiber $F_{i'}$ is $(n-1)$-connected. 
	\end{lem}
	\begin{proof}
		Notice that $X/G\vee S^n$ is the $n$-skeleton of $X/G\times S^n$. By cellular approximation, a map of pairs $(D^r,S^{r-1})\to (X/G\times S^n,X/G\vee S^n)$ is nullhomotopic for all $r\le n$, hence $\pi_r(X/G\times S^n,X/G\vee S^n)=0$ for $r\le n$, complete the proof. 
	\end{proof}
	
	Now we prove:
	
	\begin{pro}
		\label{phi X/G}
		There is a homotopy commutative diagram
		\[
		\xymatrix{
			X/G\ar[r]^(.35){\bar{\phi}}\ar@{=}[d]&X/G\vee S^n\ar[d]\\
			X/G\ar[r]^(.35){1\times\bar{\rho}}&X/G\times S^n
		}
		\]
	\end{pro}
	
	\begin{proof}
		Since $cd(X/G)=n$, the cohomology group 
		\[H^{m+1}(X/G;\underline{\pi_m(F_{i'})})=0\]
		for any $m\ge0$. Hence the obstruction class vanishes and the proof is complete. 
	\end{proof}

	\begin{thm}
		\label{lower bound coH_0}
		There is an inequality
		\[
		N\mathcal{E}(X/G)\ge d(X;\Q)=n
		\]
	\end{thm}
	
	\begin{proof}
		Let $\bar{\rho}$ and $\bar{\phi}$ be as in Propositions \ref{rho X/G2} and \ref{phi X/G}. We define a map $g\colon X/G\to X/G$ by the composition
		\[
		X/G\xrightarrow{\bar{\phi}}X/G\vee S^n\xrightarrow{1\vee q\circ\epsilon}X/G\vee X/G\xrightarrow{\nabla}X/G
		\]
		where $\nabla_Y\colon Y\vee Y\to Y$ is the folding map. Since $p_{X/G}\circ\bar{\phi}=1_{X/G}$ where $p_{X/G}:X/G\vee S^n\to X/G$ is the projection on $X/G$, $f$ is an isomorphism in $\pi_*$ for $*<n$. On the other hand,
		\begin{align*}
			g_*(q\circ\epsilon)&=\nabla\circ(1\vee q\circ\epsilon)\circ\bar{\phi}\circ q\circ\epsilon\\
			&=\nabla\circ(1\vee q\circ\epsilon)\circ(q\circ\epsilon\vee\bar{\rho}\circ q\circ\epsilon)\\
			&=\nabla\circ(1\vee q\circ\epsilon)\circ(q\circ\epsilon\vee\underline{k}\circ\rho\circ\epsilon)\\
			&=\nabla\circ(1\vee q\circ\epsilon)\circ(q\circ\epsilon+kM)\\
			&=\nabla\circ((q\circ\epsilon)\vee kM(q\circ\epsilon))\\
			&=(1+kM)(q\circ\epsilon)
		\end{align*}
		Thus since $k,M$ are positive integers and $q\circ\epsilon\in\pi_{n}(X/G)$ is of infinite order, $g$ is not an isomorphism in $\pi_{n}$. 
	\end{proof}

	\begin{proof}
		[Proof of Theorem \ref{main co-H}] 
		Apply Lemma \ref{cohomology generator} and Theorem \ref{lower bound coH_0} to \eqref{coH equation}.
	\end{proof}
	\bibliography{text,book}
\end{document}